\documentclass[12pt,reqno]{amsart}%

\usepackage[margin=1in]{geometry}
\usepackage{amsmath,amssymb,amsthm,mathtools}
\usepackage{microtype}
\usepackage{xcolor}
\usepackage[colorlinks=true,linkcolor=blue,citecolor=blue,urlcolor=blue]{hyperref}

\usepackage{amsfonts}
\usepackage{color}
\usepackage{stmaryrd}
\usepackage{graphicx}%
\providecommand{\U}[1]{\protect\rule{.1in}{.1in}}
\newtheorem{theorem}{Theorem}[section]
\theoremstyle{plain}

\newtheorem{lemma}{Lemma}[section]

\newtheorem{remark}{Remark}

\numberwithin{equation}{section}

\newcommand{\R}{\mathbb{R}}

\begin{document}
\title[Sharp constants for weak estimates of the Riesz Potentials ]{Sharp constants for weak estimates of the Riesz Potentials}

\author{Hanli Tang}
\address[Hanli Tang]{Laboratory of Mathematics and Complex Systems (Ministry of Education), School of Mathematical Sciences, Beijing Normal University, Beijing, 100875, China}
\email{hltang@bnu.edu.cn}

\author{Yaojun Wang}
\address[Yaojun Wang]{Mathematisches Institut, Ludwig-Maximilians-Universit\"at M\"unchen,
	Theresienstr. 39, 80333 M\"unchen, Germany}
\email{yaojun.wang@math.lmu.de}

\thanks{The first author
was partly supported by the National Natural Science Foundation of China (Grant No.12471100).}


\begin{abstract}
In this paper we prove the sharp weak estimates for the Riesz potentials
$$\|I_{s}(f)\|_{L^{\frac{n}{n-s},\infty}}\leq \gamma_{n,s} v_n^{\frac{n-s}{n}}\frac{\Gamma(s/2)\Gamma((n+2-s)/2)}{\Gamma(n/2)}\|f\|_{L^1},~~ \text{when}~~0<s<\min\{n,2\}$$
and
$$ \|I_sf\|_{L^{\frac{n}{n-s},\infty}}
 \leq
 \gamma_{n,s}v_n^{\frac{n-s}{n}}\|f\|_{L^1},  ~~\text{when}~~2\leq s<n,$$
where $\gamma_{n,s}=2^{-s}\pi^{-\frac{n}{2}}\frac{\Gamma(\frac{n-s}{2})}{\Gamma(\frac{s}{2})}$ and $v_n$ is the volume of the unit ball. The isoperimetric inequality for the Riesz
capacity and the Newtonian
isocapacitary inequality play a crucial role in our approach.

\end{abstract}
\maketitle
\section{Introduction}
In this paper, we are concerned with the sharp constants of weak estimates for the Riesz Potentials.
The Riesz potentials (fractional integral operator) $I_{s}$, which play an important part in harmonic analysis and partial differential equations, is defined by
$$I_{s}(f)(x)=(-\Delta)^{-s/2}f(x)=\gamma_{n,s}\int_{\mathbb{R}^n}\frac{f(x-y)}{|y|^{n-s}}dy,$$
where $0<s<n$ and $\gamma_{n,s}=2^{-s}\pi^{-\frac{n}{2}}\frac{\Gamma(\frac{n-s}{2})}{\Gamma(\frac{s}{2})}$. Such an operator
was first investigated by Frostman~\cite{F} and systematically studied by Riesz~\cite{R}. The $(L^p,L^q)$-boundedness of
Riesz potential was proved by Hardy and Littlewood~\cite{HL} when $n=1$ and by Sobolev~\cite{S} when $n>1$. The $(L^1,L^{\frac{n}{n-s},\infty})$
-boundedness was obtained by Zygmund~\cite{Z}. More precisely, they established the following theorem.

\vskip0.5cm
\textbf{Theorem A. }\textit{Let} $0<s<n$ \textit{and let}
 $p,q$ \textit{satisfy} $1\leq p<q<\infty$ \textit{and} $\frac{1}{p}-\frac{1}{q}=\frac{s}{n}$, \textit{then when} $p>1$,
$$\|I_{s}(f)\|_{L^q(\mathbb{R}^n)}\leq{C(n,p,s)}\|f\|_{L^{p}(\mathbb{R}^n)}.$$
 \textit{And when }$p=1$,
$$\|I_s(f)\|_{L^{\frac{n}{n-s},\infty}(\mathbb{R}^n)}=\sup_{\lambda>0}\lambda|\{x\in{\mathbb{R}^n}:|I_{s}f|>\lambda\}|^{\frac{n-s}{n}}\leq{C(n,s)\|f\|_{L^{1}(\mathbb{R}^n)}}.$$
\vskip0.5cm
The best constant in the strong $(L^p,L^q)$ estimate when $p=\frac{2n}{n+s}$, $q=\frac{2n}{n-s}$ was precisely calculated by Lieb \cite{L} (see also \cite{FL}), and
Lieb and Loss also obtained an upper bound of $C(n,p,s)$ (see Chapter 4 in \cite{LL}).

Although the best constants of the $(L^{\frac{2n}{n+s}},L^{\frac{2n}{n-s}})$ estimate of the Riesz potentials have been studied for decades,
there are few results about the best constants of the $(L^1,L^{\frac{n}{n-s},\infty})$ estimate for the Riesz potentials. Sometimes the study of sharp constants is highly
challenging; mathematicians first investigate the related dimension-dependent problem.

The dimension-free or dimension-dependent estimates in harmonic analysis have been a recurring theme since the pioneering work of E. M. Stein \cite{St1}, where
he proved the centered Hardy-Littlewood maximal function associated to balls
$$M_bf(x)=\sup_{r>0}\frac{1}{|B(r)|}\int_{B(r)}|f(x-y)|dy$$
admits $L^p$ bounds that are independent of the dimension $n$. Then Stein and Str\"{o}mberg \cite{SS}
proved the dimensional weak $(1,1)$ estimates for $M_b$. More precisely,
$$\|M_bf\|_{L^p(\mathbb{R}^n)}\leq C_p\|f\|_{L^p(\mathbb{R}^n)}~~\text{for}~~~1<p\leq\infty$$
and
\begin{equation}\label{weak HL est}
\|M_bf\|_{L^{1,\infty}(\mathbb{R}^n)}\leq C_1n\|f\|_{L^1(\mathbb{R}^n)},
\end{equation}
where $C_p$ and $C_1$ are independent of $n$, and whether the dependence on the dimension in (\ref{weak HL est}) is optimal remains an open question.

These results motivated extensive dimension-dependent investigations of Hardy-Littlewood maximal operators associated with more general convex bodies. We refer the interested readers to the papers of Bourgain \cite{Bou1, Bou2, Bou3}, Carbery \cite{Ca}, M\"{u}ller \cite{Mu} and Bourgain,
Mirek, Stein and Wr\'{o}bel \cite{BMSW1} and references therein. In particular, Aldaz \cite{A} established that the weak-type
$(1,1)$ bounds for the Hardy-Littlewood maximal operator associated with cubes tend to infinity as the dimension grows. Later, Iakovlev and Str\"{o}mberg \cite{IS} obtained a
quantitative lower bound $cn^{1/4}$ based on Aldaz's result.

The dimension-free or  dimension-dependent estimates for the Hardy-Littlewood maximal operator find a natural counterpart in the theory of the Riesz transform and other operators.
The Riesz transforms $R_j$ are fundamental examples of Calder\'{o}n-Zygmund singular integral operators. In 1952, Calder\'{o}n and Zygmund \cite{CZ} proved that
$$\|R_jf\|_{L^p(\mathbb{R}^n)}\leq C_p(n)\|f\|_{L^p(\mathbb{R}^n)}~~\text{for}~~~1<p<\infty$$
and
\begin{equation}\label{weak 1}
\|R_jf\|_{L^{1,\infty}(\mathbb{R}^n)}\leq C_1(n)\|f\|_{L^1(\mathbb{R}^n)}
\end{equation}
by an interpolation argument and Calder\'{o}n-Zygmund decomposition, which yields that the resulting bounds grow exponentially
with the dimension.

However, Stein \cite{St} first showed that the constants $C_p(n)$ can in fact be chosen independently of the dimension. Later Iwaniec and Martin \cite{IM},  using the method of rotations, proved that the norm of the Riesz transforms is $\cot(\frac{\pi}{2p^{\ast}})$,
which is also the norm of the Hilbert
Transform, i.e.
\begin{equation}\label{bound RT}
 \|R_j\|_{L^{p}(\mathbb{R}^n)\rightarrow L^p(\mathbb{R}^n)} =\cot(\frac{\pi}{2p^{\ast}}),~~~j=1,\cdots,n, ~ \text{where}~~p^{\ast}=\max\{p,p/(p-1)\}.
\end{equation}
A probabilistic proof of (\ref{bound RT}), based on a sharp inequality for orthogonal martingales, was later obtained by Ba\~{n}uelos and Wang \cite{BW}.

The question of whether the constant
$C_1(n)$ in (\ref{weak 1}) can be chosen independently of the dimension was solved by Ouyang, Spector and Stockdale \cite{OSS} recently, while Janakiraman \cite{Ja} (see also \cite{SpSt}) has shown that (\ref{weak 1}) holds with
 $C(n)=c\ln (n+1)$ with some absolute constant $c>0$ .

There is very little research about the best constants of weak $(L^1,L^{\frac{n}{n-s},\infty})$ estimate for the Riesz potentials,
i.e. $$\mathcal{C}_{n,s}=\sup_{0\neq f\in{L^{1}(\mathbb{R}^n)}}\frac{\|I_s(f)\|_{L^{\frac{n}{n-s},\infty}(\mathbb{R}^n)}}{\|f\|_{L^{1}(\mathbb{R}^n)}}.$$

In \cite{T} (see multilinear case in \cite{TW}), the first author set up the following  limiting weak-type behavior for the Riesz potentials
\begin{align*}
\lim_{\lambda\rightarrow{0}}\lambda|\{x\in{\mathbb{R}^n}:|I_{s}f|>\lambda\}|^{\frac{n-s}{n}}=\gamma_{n,s} v_{n}^{\frac{n-s}{n}}\|f\|_{L^{1}(\mathbb{R}^n)}~\text{for}~~ 0\leq f\in L^{1}(\mathbb{R}^n),f\not \equiv0 ,
\end{align*}
which implies
\begin{align*}
\mathcal{C}_{n,s}\geq \gamma_{n,s}v_{n}^{\frac{n-s}{n}}.
\end{align*}

In order to offer upper and lower bounds of $\mathcal{C}_{n,s}$ and  to study the behavior of $\mathcal{C}_{n,s}$ as $s\rightarrow 0$, Huang and the first author \cite{HT} consider the weak estimates of $I_{s}(f)$ under the weak norm $\interleave \cdot \interleave_{L^{\frac{n}{n-s},\infty}}$, which is equivalent to $\|\cdot\|_{L^{\frac{n}{n-s},\infty}}$ and defined by
$$\interleave f \interleave_{L^{\frac{n}{n-s},\infty}(\mathbb{R}^n)}=\sup_{0<|E|<\infty}|E|^{-\frac{1}{r}+\frac{n-s}{n}}\left(\int_E|f|^rdx\right)^{\frac{1}{r}},~~1\leq r<\frac{n}{n-s}.$$
They can establish an identity for the weak type
estimates for the Riesz potentials for $0\leq f\in L^1(\mathbb{R}^n)$ and
$0<r<\frac{n}{n-s}$
\begin{align}\label{identity}
\interleave I_{s}(f) \interleave_{L^{\frac{n}{n-s},\infty}(\mathbb{R}^n)}=\gamma_{n,s} v_{n}^{\frac{n-s}{n}}\left(\frac{n}{n-(n-s)r}\right)^{\frac{1}{r}}\|f\|_{L^1(\mathbb{R}^n)} .
\end{align}
Huang and the first author \cite{HT} also derived the optimal order-dependent estimate.

\vskip0.5cm
\textbf{Theorem B. }\textit{When} $n>2$ \textit{and} $0<s<\frac{n-2}{4}$,
 $$\gamma_{n,s} v_n^{\frac{n-s}{n}}\frac{n-2-4s}{2s(n-2-s)}\leq \mathcal{C}_{n,s}\leq \gamma_{n,s} v_n^{\frac{n-s}{n}}\frac{n }{s}.$$

\vskip0.5cm
Although the above result is optimal when $s\rightarrow 0$ for fixed $n$, these estimates did not determine the
sharp dependence on the dimension or the
explicit value of $\mathcal{C}_{n,s}$,  which will be answered in this paper.

Now let us state our main theorem. We obtain the sharp constants of weak estimates for the Riesz Potentials for all $0<s<n$.

\vskip0.5cm
\begin{theorem}\label{thm:main}
For $0<s<\min\{n,2\}$,
$$\mathcal{C}_{n,s}=\gamma_{n,s} v_n^{\frac{n-s}{n}}\frac{\Gamma(s/2)\Gamma((n+2-s)/2)}{\Gamma(n/2)}.$$
Moreover, the sharp constant can be obtained by the functions
$$c(R^2-|y|^2)^{-s/2}\chi_{B_R}(y)$$
for any $0\neq c\in \mathbb{R}$ and $R>0$. For $2\leq s<n$,
$$\mathcal{C}_{n,s}=\gamma_{n,s} v_n^{\frac{n-s}{n}}.$$
And the sharp constant can be obtained by any nonzero nonnegative $L^1$ function.
\end{theorem}
\begin{remark}
In fact, we proved that, for $2\leq s<n$, $$\|I_sf\|_{L^{\frac{n}{n-s},\infty}}
 =\gamma_{n,s}v_n^{\frac{n-s}{n}}\|f\|_{L^1}$$
for any nonzero nonnegative $L^1$ function, which is analogous to the identity \eqref{identity} and makes the lower bound $\gamma_{n,s} v_n^{\frac{n-s}{n}}$, obtained by
the first author in \cite{T}, the sharp constant.
\end{remark}
\vskip0.5cm
Let us briefly describe the main ideas and difficulties of the proof.
For $0<s<\min\{2,n\}$, let $f>0$ and $E_\lambda=\{I_sf>\lambda\}$, then the
definition of the Riesz capacity gives
$\operatorname{Cap}_s(E_\lambda)\leq\frac{\|f\|_{L^1}}{\lambda}$ (see the definition of the Riesz
capacity in Section 2.2).
The sharp weak estimate therefore follows from the isocapacitary
inequality and the explicit capacity of a ball. The key point in
proving the isocapacitary inequality is to rearrange the truncated
potential $\min\{I_sf,1\}$ of an admissible source, rather than the
source itself. The truncation estimate, together with the fractional
P\'olya--Szeg\H{o} inequality, reduces the problem to the equilibrium
potential of a ball.

At $s=2$, truncation and the classical P\'olya--Szeg\H{o} inequality
remain valid for the Dirichlet energy. However, the explicit density
$(R^2-|y|^2)^{-\frac{s}{2}}\chi_{B_R}(y)$ used in the preceding argument
is no longer integrable, so that argument cannot be extended to the
endpoint without modification. For $s>2$, the same energy-decreasing
truncation and rearrangement argument is no longer directly available.
We treat the entire range $2\leq s<n$ by a different argument based on
Newtonian capacity. For nonnegative smooth compactly supported $f$,
we show that
$V=(I_sf)^{\frac{n-2}{n-s}}$
is superharmonic and compute the total mass of $-\Delta V$ from its
behavior at infinity. The capacitary weak-type estimate and the sharp
Newtonian isocapacitary inequality then yield the upper bound. Thus the limiting weak-type formula in \cite{T} shows that every
nonzero nonnegative $L^1$ function is a maximizer.

This paper is organized as follows. Section 2 is devoted to the sharp constant for the weak estimates of the Riesz potentials when $0<s<\min\{2,n\}$. In Section
3, we will give the sharp constant for $2\leq s<n$. In Section 4, we present the proof of an asymptotically optimal estimate for $\mathcal{C}_{n,s}$
via the sharp Sobolev inequality. This estimate was the starting point of the project.

\section*{Acknowledgements}
The initial aim of this project was to obtain asymptotically optimal weak estimates for Riesz potentials and improve the result of L. Huang and the first author in \cite{HT}. The first author found an approach using Sobolev inequalities that gives a better upper bound $C\gamma_{n,s}v_n^{\frac{n-s}{n}}\frac{n^{1-s/2}}{s}$ than that in \cite{HT} (see Appendix), and then consulted AI tools to inquire whether this upper bound might be optimal. The AI tools suggested a test function $A_{n,s}(1-|y|^2)^{-s/2}\chi_{B_1}(y)$, noting that it is the density of the equilibrium measure for balls with respect to the Riesz capacity. This observation made the first author realize that a stronger result could be achieved by establishing an isoperimetric inequality for the Riesz capacity. After the first version of the paper was posted on arXiv, the second author provided a method to obtain the sharp constant for weak estimates of $I_s$ for $2\leq s<n$, so we wrote the second version of the paper. We acknowledge the use of AI tools during the exploratory stage of this project; all mathematical arguments and proofs in the final manuscript were checked and written by ourselves. The first author also thanks X. Fu and J. Li for valuable discussions. The second author acknowledges support from
the China Scholarship Council and Ludwig-Maximilians-Universit\"{a}t  M\"{u}nchen.

\section{\texorpdfstring{Sharp weak estimates for the Riesz potentials when $0<s<\min\{2,n\}$}{Sharp weak estimates for the Riesz potentials when 0<s<min(2,n)}}
\subsection{The lower bound}
In this section, we provide the lower bound for $\mathcal{C}_{n,s}$ when $0<s<\min\{2,n\}$. In fact we prove that the functions
\begin{equation*}
 \rho_R(y)=A_{n,s}(R^2-|y|^2)^{-s/2}\chi_{B_R}(y),
 \qquad
 A_{n,s}=\frac{\sin(\pi s/2)\Gamma(n/2)}
 {\gamma_{n,s}\pi^{1+n/2}}
\end{equation*}
satisfy
$$\|I_{s}\rho_R\|_{L^{\frac{n}{n-s},\infty}}\geq \gamma_{n,s} v_n^{\frac{n-s}{n}}\frac{\Gamma(s/2)\Gamma((n+2-s)/2)}{\Gamma(n/2)}\|\rho_R\|_{L^1}.$$
First we need the following lemma, which in fact is a classical theorem about explicit equilibrium measures for balls in Riesz
potential theory, see \cite{La}. For the completeness of the paper, we provide a proof with details.
\vskip0.5cm
\begin{lemma}
\label{lem:ball-potential}
For every $R>0$,
\begin{equation}
 I_s\rho_R(x)=1
 \qquad\text{when }|x|\leq R.
\label{eq:ball-potential-one}
\end{equation}
\end{lemma}
\vskip0.5cm
\begin{proof}
It
suffices to prove the lemma for $R=1$ since \eqref{eq:ball-potential-one} is invariant under the scaling $x=Ra$, $y=Rz$. Denote
\begin{equation*}
 J(a)=\int_{B_1}|a-y|^{s-n}(1-|y|^2)^{-s/2}dy.
\end{equation*}
First we assume $0<|a|<1$.  Let $\Phi_a$ be inversion in the sphere centered
at $a/|a|^2$ with squared radius $(1-|a|^2)/|a|^2$. By the change of variables $y=\Phi_a(z)$ we have
\begin{equation*}
 J(a)=(1-|a|^2)^{s/2}
 \int_{B_1}|z|^{s-n}(1-|z|^2)^{-s/2}D_a(z)^{-n/2}d z,
\end{equation*}
where
\begin{equation*}
 D_a(z)=1-2a\cdot z+|a|^2|z|^2.
\end{equation*}
By rotation invariance we may assume $a=re_1$.  In polar coordinates
$z=t\omega$, the Poisson-kernel identity yields
\begin{equation*}
 \int_{ \mathbb{S}^{n-1}}D_a(t\omega)^{-n/2}d\omega
 =\frac{|\mathbb{S}^{n-1}|}{1-r^2t^2}.
\end{equation*}
Consequently,
\begin{align}
 J(a)
 &=(1-r^2)^{s/2}|\mathbb S^{n-1}|
 \int_0^1\frac{t^{s-1}(1-t^2)^{-s/2}}{1-r^2t^2}d t
 \notag\\
 &=\frac{|\mathbb S^{n-1}|}{2}B(s/2,1-s/2)
 \notag\\
 &=\frac{\pi^{1+n/2}}
 {\Gamma(n/2)\sin(\pi s/2)},
\label{eq:constant-bare-potential}
\end{align}
where for the second equality we use Euler's integral representation of
${}_2F_1$ and the identity
${}_2F_1(1,s/2;1;r^2)=(1-r^2)^{-s/2}$. The case $a=0$ follows from the same
beta integral.

It remains to prove the case when $|a|=1$.  Fix $a\in\partial B_1$ and let
$y=a-r\theta$.  If $\tau=a\cdot\theta$, then $y\in B_1$ is equivalent to
 $0<r<2\tau$. Thus
\begin{align*}
 J(a)
 &=\int_{\{\theta:a\cdot\theta>0\}}
 \int_0^{2\tau}r^{s/2-1}(2\tau-r)^{-s/2}d rd\theta
 \notag\\
 &=\frac{|\mathbb S^{n-1}|}{2}B(s/2,1-s/2),
\end{align*}
which agrees with \eqref{eq:constant-bare-potential}.  Thus
$\gamma_{n,s}A_{n,s}J(a)=1$ for all $|a|\leq1$, which completes the proof of Lemma \ref{lem:ball-potential}.
\end{proof}
Now by Lemma \ref{lem:ball-potential} and the fact
\begin{equation*}
 \int_{B_R}(R^2-|y|^2)^{-s/2}d y
 =\frac{\pi^{n/2}\Gamma(1-s/2)}
 {\Gamma((n-s+2)/2)}R^{n-s},
\end{equation*}
we have
$$\mathcal{C}_{n,s}\geq\frac{\|I_{s}\rho_R\|_{L^{\frac{n}{n-s},\infty}}}{\|\rho_R\|_{L^1}}\geq \frac{|B_R|^{\frac{n-s}{n}}}{\|\rho_R\|_{L^1}}= \gamma_{n,s} v_n^{\frac{n-s}{n}}\frac{\Gamma(s/2)\Gamma((n+2-s)/2)}{\Gamma(n/2)}.$$

\subsection{An isoperimetric inequality for the Riesz
capacity}
In this subsection, we set up the isoperimetric inequality for the Riesz
capacity which plays a crucial role in the proof of our Main Theorem when $0<s<\min\{2,n\}$.

For any set
$E\subset\mathbb{R}^n$, define the Riesz capacity
\begin{equation*}
 \begin{aligned}
 \operatorname{Cap}_s(E)
 :=\inf\bigl\{\|f\|_{L^1(\mathbb{R}^n)}:\;&
 f\in L^1(\mathbb{R}^n),\ f\geq0, I_sf\geq1 \textit{ on } E\bigr\},
 \end{aligned}
 \label{eq:function-source-capacity-intro}
\end{equation*}
which has been systematically studied in the literature (see, for instance, \cite{AH}). The Riesz capacity has found wide applications and we refer the
interested readers to the papers \cite{OP} \cite{MV} \cite{NR} and \cite{JLSX} and references  therein. It is easy to check that
$$\operatorname{Cap}_s(rE)=r^{n-s}\operatorname{Cap}_s(E).$$

Set
\begin{equation*}
 \kappa_{n,s} =\frac{s2^{s-1}\Gamma((n+s)/2)}
 {\pi^{n/2}\Gamma(1-s/2)}.
\end{equation*}
For $u,v\in\dot H^{s/2}(\mathbb{R}^n)$, define
$$\begin{aligned}
\mathcal{E}_s(u,v)&=(2\pi)^{s}\int_{\mathbb{R}^n}|\xi|^s
 \widehat u(\xi)\overline{\widehat v(\xi)}d\xi\\
&=\frac{\kappa_{n,s}}{2}\iint_{\mathbb{R}^n\times\mathbb{R}^n}\frac{(u(x)-u(y))\overline{(v(x)-v(y))}}{|x-y|^{n+s}}dxdy,
\end{aligned}$$
 where $\widehat u(\xi)$ is the Fourier transform of $u$ definded by
$$\widehat u(\xi)=\int_{\mathbb{R}^n}
\exp^{-2\pi i x\cdot\xi}u(x)dx.$$

We write $\mathcal{E}_s(u)=\mathcal{E}_s(u,u)$ and the sharp fractional Sobolev inequality gives
$$\|u\|^2_{L^{\frac{2n}{n-s}}}\leq \mathcal{S}_{n,s}\mathcal{E}_s(u),$$
with
$$\mathcal{S}_{n,s}=\frac{\Gamma(\frac{n-s}{2})}{2^s\pi^{s/2}\Gamma(\frac{n+s}{2})}\left(\frac{\Gamma(n)}{\Gamma(n/2)}\right)^{s/n}.$$
Our isoperimetric inequality for the Riesz
capacity is the following.
\vskip0.5cm
\begin{theorem}\label{isoperimetric}
Let $0<s<\min\{2,n\}$ and let $E\subset \mathbb{R}^n$ be measurable with  $0<|E|<\infty$. If $B_R$ is a ball
such that $|B_R|=|E|$, then
$$\operatorname{Cap}_s(E)\geq \operatorname{Cap}_s(B_R).$$
Moreover,
$$\operatorname{Cap}_s(B_R)=\frac{\Gamma(n/2)}{\gamma_{n,s}\Gamma(s/2)\Gamma((n-s+2)/2)}R^{n-s}.$$
\end{theorem}
\vskip0.5cm
To prove the theorem, we need three lemmas. For $t\geq 0$, let
\begin{equation*}
 T(t)=\min\{t,1\}.
\end{equation*}
Our first lemma states
\begin{lemma}\label{lem:truncation}
If $f\in L^1(\mathbb{R}^n)$, $f\geq 0$ and $u=I_sf$, then
\begin{equation*}
 T(u)\in\dot H^{s/2}(\mathbb{R}^n),
 \qquad
 \mathcal{E}_s(T(u))\leq\|f\|_{L^1(\mathbb{R}^n)}.
\end{equation*}
\end{lemma}
\begin{proof}

Suppose first that $f\geq0$ is bounded and compactly supported. Since $f\in L^1\cap L^2$,
\begin{equation*}
 \mathcal{E}_s(u)
 =(2\pi)^{-s}\int_{\mathbb{R}^n}|\xi|^{-s}|\widehat f(\xi)|^2d\xi<\infty.
\end{equation*}
Indeed, the integral near the origin can be controlled by the fact $\widehat f(\xi)$ is bounded and $s<n$, while the integral on the region $|\xi|>1$ can be controlled by Plancherel's theorem. Thus $u\in\dot H^{s/2}(\mathbb{R}^n)$ and for any $\phi\in C_{c}^{\infty}(\mathbb{R}^n)$ there holds
\begin{equation}\label{weak-riesz-equation}
\mathcal{E}_{s}(u,\phi)=\int_{\mathbb{R}^n}f\overline{\phi} dx
\end{equation}
by the Fourier transformation. The scalar inequality
\begin{equation*}
 |T(a)-T(b)|^2
 \leq(a-b)(T(a)-T(b)),
 \qquad a,b\geq0,
 \end{equation*}
implies $T(u)\in \dot H^{s/2}(\mathbb{R}^n)$ and
$$\mathcal{E}_{s}(T(u))\leq \mathcal{E}_{s}(u,T(u)).$$
Since the right-hand side of \eqref{weak-riesz-equation} extends continuously to $\dot H^{s/2}$ because
$f\in L^{2n/(n+s)}$, by density, \eqref{weak-riesz-equation} remains
valid for $\phi=T(u)$. Since $0\leq T(u)\leq 1$, it follows that
$$\mathcal{E}_{s}(T(u))\leq \int_{\mathbb{R}^n}fT(u) dx\leq \|f\|_{L^1(\mathbb{R}^n)}.$$
For a general $0\leq f\in L^1(\mathbb{R}^n)$, we only need to take $f_j=\min\{f(x),j\}\chi_{B_j}$ and apply a classical approximation argument.
\end{proof}

Our second lemma states
\begin{lemma}
\label{lem:riesz-energy-identity}
For $0\leq h\in L^1(\mathbb{R}^n)$,
\begin{equation*}
 \int_{\mathbb{R}^n}hI_sh d x
 =(2\pi)^{-s}\int_{\mathbb{R}^n}|\xi|^{-s}|\widehat h(\xi)|^2 d\xi
\end{equation*}
with both sides allowed to be $+\infty$.
\end{lemma}
\begin{proof}
Let $p_t(x)=(4\pi t)^{-n/2}e^{-|x|^2/(4t)}$.  Since
\begin{equation}
 (2\pi|\xi|)^{-s}=\frac1{\Gamma(s/2)}
 \int_0^\infty t^{s/2-1}e^{-4\pi^2t|\xi|^2} dt,
 \label{eq:negative-power-heat}
\end{equation}
Fourier inversion for the Gaussian yields
\begin{align*}
 &(2\pi)^{-s}\int_{\mathbb{R}^n}|\xi|^{-s}|\widehat h(\xi)|^2 d\xi\\
 &\quad=\iint_{\mathbb{R}^n\times\mathbb{R}^n}
 \left(\frac1{\Gamma(s/2)}\int_0^\infty
 t^{s/2-1}p_t(x-y)d t\right)h(x)h(y)d x d y.
\end{align*}
The expression in parentheses is $\gamma_{n,s}|x-y|^{s-n}$, which proves
lemma \ref{lem:riesz-energy-identity}.
\end{proof}

Recall \begin{equation*}
 \rho_R(y)=A_{n,s}(R^2-|y|^2)^{-s/2}\chi_{B_R}(y),
 \qquad
 A_{n,s}=\frac{\sin(\pi s/2)\Gamma(n/2)}
 {\gamma_{n,s}\pi^{1+n/2}},
\end{equation*}
and denote $M_R=\|\rho_R\|_{L^1}=\frac{\Gamma(n/2)}{\gamma_{n,s}\Gamma(s/2)\Gamma((n-s+2)/2)}R^{n-s}$ and
\begin{equation*}
 e_R=I_s\rho_R.
\end{equation*}
It is easy to check $e_R\in L^{\frac{2n}{n-s}}$.

\begin{lemma}
\label{lem:ball-calibration}
$e_R$ belongs to $\dot H^{s/2}(\mathbb{R}^n)$ and
\begin{equation*}
 \mathcal{E}_s(e_R)=M_R.
\end{equation*}
If $u\in\dot H^{s/2}(\mathbb{R}^n)$ and $u=1$ almost everywhere on $B_R$, then
\begin{equation*}
 \mathcal{E}_s(u)=M_R+\mathcal{E}_s(u-e_R)\geq M_R.
\end{equation*}
Consequently,
\begin{equation*}
 \operatorname{Cap}_s(B_R)=M_R.
\end{equation*}
\end{lemma}
\begin{proof}
Since $e_R=1$ on the support of $\rho_R$, lemma \ref{lem:riesz-energy-identity}
gives
\begin{equation*}
 (2\pi)^{-s}\int_{\mathbb{R}^n}|\xi|^{-s}
 |\widehat{\rho_R}(\xi)|^2 d\xi
 =\int_{\mathbb{R}^n}\rho_Re_R dx=M_R,
\end{equation*}
which implies $e_R \in \dot H^{s/2}(\mathbb{R}^n)$ and $\mathcal{E}_s(e_R)=M_R.$

For $m\geq2$, set
\begin{equation*}
 \rho_m=\rho_R\chi_{B_{R(1-1/m)}},
 \qquad
 h_m=\rho_R-\rho_m,
\end{equation*}
thus $I_sh_m\leq e_R=1$ on $B_R$ by lemma \ref{lem:ball-potential}. Hence lemma \ref{lem:riesz-energy-identity} implies
\begin{equation}
 \|h_m\|_{\dot H^{-s/2}}^2
 =\int_{\mathbb{R}^n}h_mI_sh_m d x
 \leq\|h_m\|_1\longrightarrow0,
 \label{eq:rho-negative-sobolev-convergence}
\end{equation}
which means $\rho_m\rightarrow\rho_R$ in $\dot H^{-s/2}$. Thus for any $u\in\dot H^{s/2}(\mathbb{R}^n)$ with $u=1$ almost everywhere on $B_R$, there holds
$$
 \langle\rho_m,u\rangle=\int\rho_m d x,
$$
and \begin{equation*}
 \langle\rho_R,u\rangle=M_R
 =\langle\rho_R,e_R\rangle
\end{equation*}
by \eqref{eq:rho-negative-sobolev-convergence}.
Since $$\mathcal{E}_s(e_R,\varphi)=\langle\rho_R,\varphi\rangle$$ for
$\varphi\in\dot H^{s/2}$, it follows that
$\mathcal{E}_s(e_R,u-e_R)=0$ which implies
\begin{equation}\label{est E}
 \mathcal{E}_s(u)=M_R+\mathcal{E}_s(u-e_R)\geq M_R.
\end{equation}
Now, we only need to prove
\begin{equation*}
 \operatorname{Cap}_s(B_R)=M_R.
\end{equation*}
It is obvious that $\operatorname{Cap}_s(B_R)\leq \|\rho_R\|_{L^1}=M_R$. On the other hand, for any $0\leq g\in L^{1}$ with $I_sg\geq 1$ on $B_R$, we have that
$$\|g\|_{L^1}\geq \mathcal{E}_s(T(I_sg))\geq M_R$$
by lemma \ref{lem:truncation} and \eqref{est E}, which completes the proof of the lemma.
\end{proof}
Now we are in a position to prove the theorem \ref{isoperimetric}. Let $0\leq f\in L^{1}$ and $I_sf\geq 1$ on $E$. Denote
$$w=\min\{I_sf,1\}.$$
Thus we have
\begin{equation*}
 w\in\dot H^{s/2}(\mathbb{R}^n),
 \qquad
 \mathcal{E}_s(w)\leq\|f\|_1
\end{equation*}
by lemma \ref{lem:truncation}. And we also have $w=1$ on $E$. Let $w^{\ast}$ be the rearrangement of $w$. Then $w^{\ast}=1$ a.e. on the ball $B_R$ with measure $|B_R|=|E|$ and
$$\mathcal{E}_s(w^{\ast})\leq \mathcal{E}_{s}(w)$$
by the fractional P\'{o}lya-Szeg\H{o} inequality for $0<s<2$. Therefore, by lemma \ref{lem:truncation} and lemma \ref{lem:ball-calibration} we obtain
$$\|f\|_{L^1}\geq\mathcal{E}_{s}(w)\geq\mathcal{E}_s(w^{\ast})\geq M_R=\operatorname{Cap}_{s}(B_R),$$
which implies
$$\operatorname{Cap}_{s}(E)\geq \operatorname{Cap}_{s}(B_R).$$

\subsection{Sharp weak estimates for the Riesz potentials}
Now we can prove our main theorem when $0<s<\min\{2,n\}$. Since $|I_sf|\leq I_s|f|$, then we may assume that $f\geq 0$. For $\lambda>0$, set
$E_{\lambda}=\{I_sf(x)\geq \lambda\}$. Thus
$$\operatorname{Cap}_s(E_\lambda)\leq \frac{\|f\|_{L^1}}{\lambda}$$
and the isoperimetric inequality for Riesz capacity namely Theorem \ref{isoperimetric} implies
$$(\frac{|E_\lambda|}{v_n})^{\frac{n-s}{n}}\operatorname{Cap}_s(B_1)\leq \operatorname{Cap}_s(E_\lambda).$$
Since $$\operatorname{Cap}_s(B_1)=M_1=\frac{\Gamma(n/2)}{\gamma_{n,s}\Gamma(s/2)\Gamma((n-s+2)/2)}$$
by lemma \ref{lem:ball-calibration}, then we have
$$\lambda|E_\lambda|^{\frac{n-s}{n}}\leq \frac{v_n^{\frac{n-s}{n}}\|f\|_{L^1}}{\operatorname{Cap}_s(B_1)}=\gamma_{n,s}v_n^{\frac{n-s}{n}}\frac{\Gamma(s/2)\Gamma((n-s+2)/2)}{\Gamma(n/2)}\|f\|_{L^1}.$$
On the other hand, in Section 2.1 we already proved that the sharp constant can be attained by the function
\begin{equation*}
 \rho_R(y)=A_{n,s}(R^2-|y|^2)^{-s/2}\chi_{B_R}(y),
 \qquad
 A_{n,s}=\frac{\sin(\pi s/2)\Gamma(n/2)}
 {\gamma_{n,s}\pi^{1+n/2}}.
\end{equation*}

\section{\texorpdfstring{Sharp weak estimates for the Riesz potentials when $2\leq s<n$}{Sharp weak estimates when 2<=s<n}}\label{sec:weak}
Throughout this section, we assume $n\geq3$ and $2\leq s<n$.

\begin{theorem}\label{thm:large-s}
For every $f\in L^1(\mathbb R^n)$,
\begin{equation}\label{eq:large-s-weak}
 \|I_sf\|_{L^{\frac{n}{n-s},\infty}(\mathbb R^n)}
 \leq
 \gamma_{n,s}v_n^{\frac{n-s}{n}}\|f\|_{L^1(\mathbb R^n)}.
\end{equation}
If $f\geq0$ and $f\not\equiv0$, then
\begin{equation}\label{eq:large-s-limit}
 \lim_{\lambda\downarrow0}
 \lambda|\{I_sf>\lambda\}|^{\frac{n-s}{n}}
 =
 \gamma_{n,s}v_n^{\frac{n-s}{n}}\|f\|_{L^1(\mathbb R^n)}.
\end{equation}
Consequently, every nonzero nonnegative $L^1$ function is a maximizer in
\eqref{eq:large-s-weak}, and the constant in \eqref{eq:large-s-weak}
is sharp.
\end{theorem}

The limiting identity \eqref{eq:large-s-limit} was proved by the first author in \cite{T}.
To prove the upper bound, we first show that
$(I_sf)^{\frac{n-2}{n-s}}$ is superharmonic and compute the total mass
of its negative Laplacian. We then combine the capacitary weak-type
estimate with the sharp Newtonian isocapacitary inequality.

We first recall the homogeneous Sobolev space used below.  Following
\cite[Section~8.2]{LL}, we denote by $D^{1,2}(\mathbb{R}^n)$ the space
denoted there by $D^1(\mathbb{R}^n)$: it consists of all
$w\in L^1_{\mathrm{loc}}(\mathbb{R}^n)$ whose distributional gradient
belongs to $L^2(\mathbb{R}^n)$ and which vanish at infinity in the sense
that
$$
 \bigl|\{x\in\mathbb{R}^n:|w(x)|>a\}\bigr|<\infty,\qquad\text{for every }a>0.
$$

\begin{lemma}\label{lem:newtonian-power}
Let $0\leq f\in C_c^\infty(\mathbb{R}^n)$, $V:=(I_sf)^{\frac{n-2}{n-s}}$, $f\not\equiv0$, and
$$
 m:=\int_{\mathbb{R}^n}f(x)dx.
$$
Then $V\in C^\infty(\mathbb{R}^n)\cap D^{1,2}(\mathbb{R}^n)$ is positive,
$V(x)\to0$ as $|x|\to\infty$, and $-\Delta V\geq0$.  Moreover,
\begin{equation}\label{eq:newtonian-mass}
 \int_{\mathbb{R}^n}(-\Delta V)dx
 =n(n-2)v_n(\gamma_{n,s}m)^{\frac{n-2}{n-s}}.
\end{equation}
\end{lemma}

\begin{proof}
Set $u:=I_sf$.
The positivity of the kernel and standard convolution regularity show that
$u$ is positive and smooth, and therefore so is $V$.  Since $s\geq2$,
differentiating the kernel gives
\begin{equation}\label{eq:large-s-gradient-u}
 \nabla u(x)
 =-(n-s)\gamma_{n,s}
 \int_{\mathbb{R}^n}(x-y)|x-y|^{s-n-2}f(y)dy.
\end{equation}
If $s=2$, then
$$
 V=I_2f=(-\Delta)^{-1}f,
 \qquad
 -\Delta V=f\geq0.
$$
Suppose that $2<s<n$.  The operator identity
$I_s=(-\Delta)^{-\frac{s}{2}}$ gives $ -\Delta u=I_{s-2}f$, which first holds in $\mathcal{S}'(\mathbb{R}^n)$ and then pointwise, since both sides are smooth.

Since $\gamma_{n,s-2}=(n-s)(s-2)\gamma_{n,s}$, using the Cauchy--Schwarz inequality,
\begin{equation}\label{eq:gradient}
	\begin{split}
		|\nabla u(x)|^2 &= (n-s)^2 \gamma_{n,s}^2\left|\int_{\R^n} |x-y|^{\frac{s-n}{2}}f(y)^{\frac{1}{2}} (x-y)|x-y|^{\frac{s-n}{2}-2}f(y)^{\frac{1}{2}}dy\right|^2\\
		&\leq (n-s)^2\gamma_{n,s}u(x)\int_{\R^n} |x-y|^{s-n-2}f(y)dy\\
		&=\frac{n-s}{s-2}\,u(x)I_{s-2}f(x).
	\end{split}
\end{equation}
Since $V=u^{\frac{n-2}{n-s}}$, the chain rule, the identity
$-\Delta u=I_{s-2}f$, and \eqref{eq:gradient} yield
$$
-\Delta V=\frac{n-2}{n-s}u^{\frac{n-2}{n-s}-2}
\left[uI_{s-2}f-\frac{s-2}{n-s}|\nabla u|^2
\right]\geq 0.
$$

It remains to verify the decay and \eqref{eq:newtonian-mass}. A direct computation gives the following estimates which hold uniformly in the direction of $x$:
\begin{align*}
	u(x)&=\gamma_{n,s}m|x|^{s-n}+O(|x|^{s-n-1}),\\
	\nabla u(x)&=-(n-s)\gamma_{n,s}m x|x|^{s-n-2}+O(|x|^{s-n-2}).
\end{align*}
The first estimate gives $V(x)=O(|x|^{2-n})$, and hence $V(x)\to 0$ as $|x|\to \infty$ and
$V\in L^{\frac{2n}{n-2}}(\R^n)$. Applying the chain rule to the two
estimates gives
$$
\nabla V(x)=-(n-2)(\gamma_{n,s}m)^{\frac{n-2}{n-s}}
x|x|^{-n}
+O(|x|^{-n}).
$$
Thus $|\nabla V(x)|=O(|x|^{1-n})$. These decay estimates imply $\nabla V\in L^2(\R^n)$ and $V\in D^{1,2}(\R^n)$.

Moreover, using that $\partial_r=\frac{x}{|x|}\cdot\nabla$ we obtain
$$
\partial_rV(x)
=-(n-2)(\gamma_{n,s}m)^{\frac{n-2}{n-s}}
|x|^{1-n}
+O(|x|^{-n}).
$$
The divergence theorem on $B_R$ therefore yields
$$
 \int_{B_R}(-\Delta V)\,dx=-\int_{\partial B_R}\partial_rV\,d\sigma=n(n-2)v_n(\gamma_{n,s}m)^{\frac{n-2}{n-s}}
 +O(R^{-1}).
$$
Since $-\Delta V\geq0$, letting $R\to\infty$ and using monotone convergence proves \eqref{eq:newtonian-mass}.
\end{proof}

\begin{lemma}\label{lem:newtonian-level-set}
	Let $n\geq3$ and $W\in C^\infty(\mathbb R^n)\cap D^{1,2}(\mathbb R^n)$ be positive,
	assume that $W(x)\to0$ as $|x|\to\infty$ and $-\Delta W\geq0$, and
	suppose that, for some $M>0$,
	$$
	\int_{\mathbb R^n}(-\Delta W)\,dx=n(n-2)v_nM.
	$$
	Then for every $t>0$,
	\begin{equation}\label{eq:newtonian-level-set}
		t|\{W>t\}|^{\frac{n-2}{n}}
		\leq v_n^{\frac{n-2}{n}}M.
	\end{equation}
\end{lemma}

\begin{proof}
	For a bounded measurable set $E\subset\mathbb R^n$, we define the following capacity
	$$
	\operatorname{cap}_2(E):=\inf\left\{
	\int_{\mathbb R^n}|\nabla\varphi|^2\,dx:
	\varphi\in D^{1,2}(\mathbb R^n),\
	\varphi\text{ is continuous},\
	\varphi\geq1\text{ on }E
	\right\}.
	$$
	See more information in \cite[Section~11.15, formula~(7)]{LL}.
	
	We claim the following isocapacitary inequality:
	\begin{equation}\label{eq:sharp-newtonian-isocapacitary}
		\operatorname{cap}_2(E) \geq n(n-2)v_n^{\frac{2}{n}}|E|^{\frac{n-2}{n}}.
	\end{equation}
	Indeed, suppose $|E|>0$, let $v_nR^n=|E|$, and take an admissible
	$\varphi$. Set
	$$
	w=\min\{\max\{\varphi,0\},1\}.
	$$
	Then $w=1$ on $E$ and
	$\int|\nabla w|^2\leq\int|\nabla\varphi|^2$.
	By the P\'olya--Szeg\H{o} inequality \cite[Lemma~7.17]{LL}, its rearrangement
	$w^\ast(x)=h(|x|)$ satisfies
	$$
	\int_{\mathbb R^n}|\nabla w^\ast|^2\,dx
	\leq\int_{\mathbb R^n}|\nabla w|^2\,dx,
	\qquad w^\ast=1\text{ almost everywhere on }B_R.
	$$
	Since $w^\ast$ is bounded and has square-integrable gradient,
	radial mollification and polar coordinates give
	$h\in H^1(a,b)$ for every $0<a<b<\infty$. Thus $h$ has a locally absolutely continuous representative
	on $(0,\infty)$, which we continue to denote by $h$.
	This representative is nonnegative and nonincreasing.
	Since $w^\ast=1$ almost everywhere on $B_R$, continuity gives
	$h(R)=1$.
	Moreover, the positive superlevel sets of $w^\ast$ have finite
	measure, so $h(r)\to0$ as $r\to\infty$.
	Consequently,
	$$
	1=\lim_{T\to\infty}\bigl(h(R)-h(T)\bigr)
	=\int_R^\infty-h'(r)\,dr.
	$$
	The Cauchy--Schwarz inequality gives
	$$
	1=\int_R^\infty-h'(r)\,dr
	\leq
	\left(\int_R^\infty|h'(r)|^2r^{n-1}\,dr\right)^{\frac{1}{2}}
	\left(\frac{R^{2-n}}{n-2}\right)^{\frac{1}{2}}.
	$$
	Therefore
	$$
	\int_{\mathbb R^n}|\nabla\varphi|^2\,dx
	\geq n v_n\int_R^\infty|h'(r)|^2r^{n-1}\,dr
	\geq n(n-2)v_nR^{n-2},
	$$
	which proves \eqref{eq:sharp-newtonian-isocapacitary}. The case $|E|=0$
	is immediate.
	
	The continuous function
	$$
	\psi(x)=\min\left\{1,\left(\frac{R}{|x|}\right)^{n-2}\right\}
	$$
    equals $1$ at the origin, belongs to
	$D^{1,2}$ and satisfies
	$$
	\int_{\R^n}|\nabla \psi(x)|^2dx =n(n-2)v_nR^{n-2}.
	$$
	Hence $\operatorname{cap}_2(B_R)=n(n-2)v_nR^{n-2}$, so the inequality is sharp.
	This is the isocapacitary inequality; see also
	\cite[Theorem~11.17]{LL}.
	
	Fix $t>0$ and set $E_t=\{W\geq t\}$. This set is compact.
Since the truncation $\varphi_t=\min\left\{\frac{W}{t},1\right\}$
	is continuous, belongs to $D^{1,2}(\mathbb R^n)$, and equals $1$ on
	$E_t$, we have
	\begin{equation}\label{eq:capacity-by-truncation}
		\operatorname{cap}_2(E_t)
		\leq
		\frac{1}{t^2}
		\int_{\{W<t\}}|\nabla W|^2\,dx.
	\end{equation}
	
	Set $\psi_t:=\min\{W,t\}$. Since
	$\psi_t\in D^{1,2}(\R^n)$ and $0\leq\psi_t\leq t$, cutoff and
	mollification give functions $\psi_k\in C_c^\infty(\R^n)$ such that
	$$
	0\leq\psi_k\leq t,
	\qquad
	\psi_k\to\psi_t\ \text{almost everywhere},
	\qquad
	\nabla\psi_k\to\nabla\psi_t\ \text{in }L^2(\R^n).
	$$
	Since $\nabla W\in L^2(\R^n)$, $-\Delta W\geq 0$ and $\int_{\R^n} -\Delta W d x <\infty$, we obtain
	\begin{align*}
		\int_{\{W<t\}}|\nabla W|^2d x&=\int_{\R^n} \nabla W\cdot \nabla \psi_td x
		=\lim_k\int_{\R^n}\nabla W\cdot\nabla\psi_kd x\\
		&=\lim_k\int_{\R^n}\psi_k(-\Delta W)d x\leq t\int_{\R^n}(-\Delta W)d x.
	\end{align*}
	Here we used the integration by parts and dominated convergence. It follows that
	$$
	\operatorname{cap}_2(E_t)
	\leq
	\frac{1}{t}\int_{\R^n}(-\Delta W)d x
	=
	\frac{n(n-2)v_nM}{t}.
	$$
	
	Combining the upper and lower capacity bounds gives
	$$
	n(n-2)v_n^{\frac{2}{n}}|E_t|^{\frac{n-2}{n}} \leq\frac{n(n-2)v_nM}{t}.
	$$
	Since $\{W>t\}\subset E_t$, this gives
	$$
	n(n-2)v_n^{\frac{2}{n}}|\{W>t\}|^{\frac{n-2}{n}}\leq\frac{n(n-2)v_nM}{t},
	$$
	which proves \eqref{eq:newtonian-level-set}.
\end{proof}

\begin{proof}[Proof of Theorem~\ref{thm:large-s}]
The case $f=0$ is immediate. Since$|I_sf|\leq I_s|f|$ almost everywhere, it suffices to consider
nonnegative $f$.
We first consider $0\leq f\in C_c^\infty(\mathbb R^n)$,
$f\not\equiv0$, and set
$$
 m:=\|f\|_{L^1(\mathbb R^n)},
 \qquad V:=(I_sf)^{\frac{n-2}{n-s}},
 \qquad M:=(\gamma_{n,s}m)^{\frac{n-2}{n-s}}.
$$
By Lemma~\ref{lem:newtonian-power}, $V$ satisfies all the assumptions of
Lemma~\ref{lem:newtonian-level-set}. Applying that lemma with
$t=\lambda^{\frac{n-2}{n-s}}$ gives
$$
 \lambda^{\frac{n-2}{n-s}}
 |\{I_sf>\lambda\}|^{\frac{n-2}{n}}
 \leq v_n^{\frac{n-2}{n}}
 (\gamma_{n,s}m)^{\frac{n-2}{n-s}}.
$$
Raising both sides to the power $\frac{n-s}{n-2}$ yields
$$
 \lambda|\{I_sf>\lambda\}|^{\frac{n-s}{n}}
 \leq
 \gamma_{n,s}v_n^{\frac{n-s}{n}}\|f\|_{L^1(\mathbb R^n)}.
$$

For general $0\leq f\in L^1(\mathbb R^n)$, choose
$0\leq f_j\in C_c^\infty(\mathbb R^n)$ such that
$\|f_j-f\|_{L^1(\mathbb R^n)}\to0$.
For every compact set $K\subset\mathbb R^n$, splitting the integral
into $|x-y|<1$ and $|x-y|\geq1$ gives
$$
 \sup_{y\in\mathbb R^n}\int_K|x-y|^{s-n}\,dx
 \leq\frac{nv_n}{s}+|K|.
$$
Consequently, Tonelli's theorem implies
$$
 \|I_s(f_j-f)\|_{L^1(K)}
 \leq
 \gamma_{n,s}\left(\frac{nv_n}{s}+|K|\right)
 \|f_j-f\|_{L^1(\mathbb R^n)}
 \longrightarrow0.
$$
After passing to a subsequence, we assume that $I_sf_j\to I_sf$ almost
everywhere. For every $\lambda>0$, Fatou's lemma gives
$$
 \lambda|\{I_sf>\lambda\}|^{\frac{n-s}{n}}
 \leq\liminf_{j\to\infty}\lambda|\{I_sf_j>\lambda\}|^{\frac{n-s}{n}}\leq \lim_{j\to\infty}\gamma_{n,s}v_n^{\frac{n-s}{n}}\|f_j\|_{L^1(\mathbb R^n)}=\gamma_{n,s}v_n^{\frac{n-s}{n}}\|f\|_{L^1(\mathbb R^n)}.
$$
Taking the supremum
over $\lambda>0$ proves \eqref{eq:large-s-weak}.
Finally, \eqref{eq:large-s-weak} and \eqref{eq:large-s-limit} give
$$
 \|I_sf\|_{L^{\frac{n}{n-s},\infty}}
 =
 \gamma_{n,s}v_n^{\frac{n-s}{n}}\|f\|_{L^1}
 \qquad(0\leq f\in L^1,\ f\not\equiv0),
$$
which proves the assertions about maximizers and sharpness.
\end{proof}

\section{Appendix}
In this section, we will provide a sharp Sobolev inequality approach by which we can obtain an asymptotically optimal estimate for $\mathcal{C}_{n,s}$.
The equivalence between the sharp Sobolev inequality and the sharp Hardy-Littlewood-Sobolev inequality (the $(L^{\frac{2n}{n+s}},L^{\frac{2n}{n-s}})$ estimate of the Riesz potentials) is classical, see \cite{L}. More recently, it has been observed that these two inequalities are also closely connected in stability theorems; see \cite{Carlen} and \cite{CLT1}. The purpose of this appendix is to add that the sharp Sobolev inequality can be used to derive an asymptotically optimal estimate for weak-type estimates of Riesz potentials, which appears to be new in the literature.

Let $0<s<\min\{2,n\}$, $0\leq f\in L^1(\mathbb{R}^n)\cap C_{c}^{\infty}(\mathbb{R}^n)$ and $v=I_sf$. Then $f=(-\Delta)^{s/2}v$. Denote $v_\lambda=\min\{v,\lambda\}$, $\lambda>0$ and
$E_\lambda=\{I_sf\geq \lambda\}=\{v\geq \lambda\}$. Since $v_{\lambda}=\lambda$ on $E_\lambda$, we have
$$|E_\lambda|\lambda^{\frac{2n}{n-s}}=\int_{E_\lambda}v_\lambda^{\frac{2n}{n-s}}dx\leq \|v_\lambda\|_{L^{\frac{2n}{n-s}}}^{\frac{2n}{n-s}}.$$
The sharp fractional Sobolev inequality states
$$\|v_\lambda\|^2_{L^{\frac{2n}{n-s}}}\leq \mathcal{S}_{n,s}\mathcal{E}_{s}(v_\lambda),$$
where $$\mathcal{S}_{n,s}=\frac{\Gamma(\frac{n-s}{2})}{2^s\pi^{s/2}\Gamma(\frac{n+s}{2})}\left(\frac{\Gamma(n)}{\Gamma(n/2)}\right)^{s/n}.$$
On the other hand,
$$\mathcal{E}_{s}(v_\lambda)\leq \mathcal{E}_{s}(v_\lambda,v)=\langle(-\Delta)^{s/2}v,v_\lambda\rangle=\int_{\mathbb{R}^n}f(x)v_\lambda(x)dx\leq \lambda\|f\|_{L^1}.$$
Thus
$$|E_\lambda|\lambda^{\frac{2n}{n-s}}\leq (\mathcal{S}_{n,s})^{\frac{n}{n-s}}\lambda^{\frac{n}{n-s}}\|f\|_{L^1}^{\frac{n}{n-s}},$$
which implies for $f\in L^1(\mathbb{R}^n)\cap C_{c}^{\infty}(\mathbb{R}^n)$ there holds
$$\lambda|E_\lambda|^{\frac{n-s}{n}}\leq \mathcal{S}_{n,s}\|f\|_{L^1}.$$
By classical approximation approach, we can obtain that
$$\mathcal{C}_{n,s}\leq \mathcal{S}_{n,s}.$$
Since $$\mathcal{S}_{n,s}=2^{s/2}\pi^{-s/2}e^{-s/2}n^{-s/2}(1+O_{s}(\frac{\ln n}{n}))$$
and
$$\gamma_{n,s}v_n^{\frac{n-s}{n}}=\frac{2^{1-s}}{\Gamma(s/2)}\pi^{-s/2}e^{-s/2}n^{-1}(1+O_s(\frac{\ln n}{n})),$$
then
$$\frac{\mathcal{S}_{n,s}}{\gamma_{n,s}v_n^{\frac{n-s}{n}}}\leq C\frac{n^{1-s/2}}{s}$$
for some positive constant $C>0$ independent of $s$ and $n$, which implies
the upper bound $\mathcal{S}_{n,s}$ is asymptotically optimal as $s\to 0^+$ for each fixed $n$, and has the optimal order as $n\to\infty$
for each fixed $0<s<2$. In fact, the explicit formula
for $\mathcal{C}_{n,s}$ gives
\[
\lim_{s\to 0^+}
\frac{\mathcal{C}_{n,s}}{\mathcal{S}_{n,s}}=1
\qquad\text{for each fixed }n\geq1,
\]
and
\[
\lim_{n\to\infty}
\frac{\mathcal{C}_{n,s}}{\mathcal{S}_{n,s}}=2^{-s}
\qquad\text{for each fixed }0<s<2.
\]

\bibliographystyle{amsalpha}

\end{document}